\documentclass[12pt]{amsart}
\usepackage{amscd,amsmath,amssymb,amsfonts}
\usepackage[all]{xy}
\usepackage{hyperref}
\usepackage{url}
\usepackage{stmaryrd}
\usepackage{color}

\theoremstyle{plain}
\newtheorem{thm}{Theorem}
\newtheorem{lem}[thm]{Lemma}
\newtheorem{cor}[thm]{Corollary}
\newtheorem{prop}[thm]{Proposition}

\newtheorem{defn}[thm]{Definition}

\theoremstyle{definition}
\newtheorem{ex}[thm]{Example}

\newtheorem{claim}[thm]{Claim}
\newtheorem{rmk}[thm]{Remark}

\numberwithin{thm}{section}

\newcommand{\ml}[2]{\begin{multline}\label{#1}#2 \end{multline}}
\newcommand{\ga}[2]{\begin{gather}\label{#1}#2 \end{gather}}

\newcommand{\codim}{{\rm codim}}

\newcommand{\Hom}{{\rm Hom}}
\newcommand{\im}{{\rm im}}
\newcommand{\Spec}{{\rm Spec \,}}

\newcommand{\ch}{{\rm ch}}

\newcommand{\Gal}{{\rm Gal}}

\newcommand{\sD}{{\mathcal D}}

\newcommand{\sF}{{\mathcal F}}
\newcommand{\sG}{{\mathcal G}}
\newcommand{\sH}{{\mathcal H}}

\newcommand{\sK}{{\mathcal K}}
\newcommand{\sL}{{\mathcal L}}
\newcommand{\sM}{{\mathcal M}}

\newcommand{\sO}{{\mathcal O}}

\newcommand{\sR}{{\mathcal R}}

\newcommand{\C}{{\mathbb C}}

\newcommand{\F}{{\mathbb F}}
\newcommand{\G}{{\mathbb G}}

\newcommand{\N}{{\mathbb N}}
\renewcommand{\P}{{\mathbb P}}
\newcommand{\Q}{{\mathbb Q}}
\newcommand{\R}{{\mathbb R}}

\newcommand{\Z}{{\mathbb Z}}

\newcommand{\QQl}{\overline{\mathbb Q}_\ell}
\newcommand{\ZZl}{\overline{\mathbb Z}_\ell}

\newcommand{\coh}{{\rm coh}}
\DeclareMathOperator{\FM}{FM}
\DeclareMathOperator{\Spm}{Spm}
\DeclareMathOperator{\colim}{colim}
\DeclareMathOperator{\sFM}{\mathfrak{FM}}
\DeclareMathOperator{\Frac}{Frac}
\DeclareMathOperator{\Spf}{Spf}

\DeclareMathOperator{\coker}{coker}

\begin{document}

\title[Rank one local systems]{  \'Etale cohomology of rank one $\ell$-adic local systems in positive characteristic
   }
\author{H\'el\`ene Esnault \and Moritz Kerz}
\address{Freie Universit\"at Berlin, Arnimallee 3, 14195, Berlin,  Germany}
\email{esnault@math.fu-berlin.de}
\address{   Fakult\"at f\"ur Mathematik \\
Universit\"at Regensburg \\
93040 Regensburg, Germany}
\email{moritz.kerz@mathematik.uni-regensburg.de}
\thanks{The first  author is supported by  the Berlin Einstein program,  the second author by the SFB 1085 Higher Invariants, Universit\"at Regensburg}
\subjclass{14G17,14G22}

\begin{abstract} We show that in positive characteristic special loci of deformation
  spaces of rank one $\ell$-adic local systems are quasi-linear.  From this we deduce the
  Hard Lefschetz theorem for rank one  $\ell$-adic local systems and  a
generic vanishing theorem.

\end{abstract}

\maketitle

\section{Introduction}

\noindent
In this note we study  the cohomology of  \'etale  rank one $\ell$-adic local systems on algebraic varieties in terms of $\ell$-adic
analysis applied to  the deformation space of local systems. One feature of our approach is
that we apply non-archimedean techniques (formal Lie groups, affinoid algebras etc.) that
were originally developed as tools in the study of
$p$-adic cohomology theories.
Before we explain our new non-archimedean methods in  Section~\ref{ss:MainThm} we sketch some applications in
the most simple form. The
general form  is described  in Sections~\ref{sec:HL} and~\ref{sec:jumvan}.

\subsection{Hard Lefschetz}\label{subsec:inthl}

One of the applications of our $\ell$-adic technique is a new case of the Hard Lefschetz
isomorphism in positive characteristic. Let $\ell$ be a prime number and let $X$ be a
smooth projective variety over an
algebraically closed field $F$ of characteristic different from $\ell$. Let $d$ be the dimension of $X$.
Let  $\eta\in H^2(X,\Z_\ell)$ be  a polarization in \'etale cohomology, i.e.\  the first Chern class of an ample line bundle on $X$, where we omit Tate twists as $F$ is algebraically closed.

\begin{thm}[Hard Lefschetz] \label{thm:hl1}
 Let $\sL$ be a  rank one  \'etale $\overline \Q_\ell$-local system on $X$.
    Then for any $i\in \N,$ the cup-product map
 \[
\cup \eta^i \colon H^{-i}(X,  \sL[d]) \xrightarrow{\sim}   H^{i}(X,  \sL[d])
\]
on \'etale cohomology
is an isomorphism.
\end{thm}

More generally, we also obtain the isomorphism $\cup \eta^i$ for $\sL[d]$ replaced by $\sF\otimes
\sL$, where $\sF\in D^b_c(X,\QQl)$ is an arithmetic semi-simple perverse sheaf, see   Theorem~\ref{thm:hl}.
Here we call a sheaf arithmetic if it is fixed by the action of a Galois group of a
finitely generated field, see
Definition~\ref{defn:arithmperv} and Remark~\ref{rmk:arithm}. In particular, semi-simple objects $\sF\in
D^b_c(X,\QQl)$ of geometric origin~\cite[p.~163]{BBD82} are arithmetic.

\begin{rmk}
In characteristic zero, Theorem~\ref{thm:hl1} in the more general case where $\sL$  is a
semi-simple local system of arbitrary rank
was first shown by complex analytic techniques
in~\cite[Lem.~2.6]{Sim92}. In fact in characteristic zero we know the Hard Lefschetz isomorphism for
any semi-simple perverse sheaf, see  \cite[ Thm.~5.4.10]{BBD82},    \cite[Thm.~1.4]{Dri01}, \cite[MainThm.~1]{Sab05} and \cite[Thm.~19.47]{Moc07}.

In positive characteristic,  Theorem~\ref{thm:hl1} was shown for a torsion rank one local system
$\sL$ by Deligne~\cite[Thm.~4.1.1]{Del80} relying on arithmetic  weight arguments. More generally the
Hard Lefschetz isomorphism for an arithmetic semi-simple perverse sheaf was known from
combining \cite[Thm.~6.2.10]{BBD82} and the Langlands correspondence for ${\rm GL}_r$ over function
fields~\cite[Thm.~7.6]{Laf02}, see~\cite[1.8]{Dri01}.
\end{rmk}

\subsection{Jumping loci}\label{subsec:intjump}

Jumping loci in the moduli space of line bundles or rank one local systems on complex
varieties have been studied extensively, see~\cite{GL91}  for the initial approach. As a second application we show that jumping loci
in the deformation space of \'etale rank one $\QQl$-local systems satisfy a strong
linearity condition analogous to the well-understood situation in characteristic zero, see
Section~\ref{subsec:jump} for details.

Let $\pi$ be a finitely generated free $\Z_\ell$-module. Let
$$\sG(\QQl) :=\Hom_{\rm cont}(\pi, \QQl^\times)$$
be the set of continuous homomorphisms $\pi\to E^\times$, where
 $E\subset \QQl$ is any finite extension of $\Q_\ell$ endowed with the $\ell$-adic topology.
 We endow $\sG(\QQl)$ with a noetherian Zariski topology such that
 $\dim(\sG(\QQl))=\rm{rank}(\pi)$. This Zariski topology originates from the observation that
 the elements of
 $\sG(\QQl)$ are the $\QQl$-points of a multiplicative formal Lie group $\sG$, see Section~\ref{subsec:formbas}.

For a torsion free quotient $\Z_\ell$-module  $\pi/\pi'$,  we call
$$\sH(\QQl)=\Hom_{\rm cont}(\pi/\pi', \QQl^\times)$$  a {\it formal Lie subgroup} of
$\sG(\QQl)$, or more precisely the $\QQl$-points of a formal Lie subgroup $\sH$.
 Note that  $\codim_{\sG(\QQl)}(\sH(\QQl))=\rm{rank}(\pi')$.

\smallskip

Let $X$ be either a smooth proper variety or the torus $\G_m^d$ over an algebraically
closed field $F$ of characteristic different from $\ell$. In the following we let the free $\Z_\ell$-module
$\pi$ be a quotient of the abelian \'etale fundamental group $\pi_1^{\rm ab}( X)$. Then
with the group of characters $\sG(\QQl)$ as above, any
$s\in \sG(\QQl)$ gives rise to an \'etale rank one $\QQl$-local system $\sL_s$ on $X$.

 For $\sF\in D^b_c(X ,\QQl)$ and $i,j \in \Z$,  we consider the jumping locus of \'etale cohomology
 \[
\Sigma^i(\sF,j) := \{ s\in \sG(\QQl) \, |\, {\rm dim} \ H^i(X ,\sF\otimes\sL_s ) > j \}
\]
which one can show to be Zariski closed in $\sG(\QQl)$.
 Our second application is a structure theorem for those loci, see Theorem~\ref{thm.jump}.

\begin{thm}\label{thm.jumpint}
  For an arithmetic  sheaf $\sF\in  D^b_c (X,\QQl )$ we have
  \[
    \Sigma^i(\sF,j) = \bigcup_{r\in I} s_r \sH_r( \QQl ) ,
  \]
  where $I$ is finite, $s_r\in \sG(\QQl)$ are torsion points and $\sH_r$ are formal
  Lie subgroups of $\sG$.

\end{thm}

\begin{rmk}
 As far as we are aware of, Theorem~\ref{thm.jumpint} is the first result on jumping loci of
 $\ell$-adic cohomology in
 positive characteristic.

For $\ch(F)=0$ and $\sF$ of geometric origin in the sense of~\cite[p.~163]{BBD82},
Theorem~\ref{thm.jumpint} is shown in \cite[Thm.~0.1]{GL91}, \cite[Thm.~4.2]{Sim93},  \cite[Thm.~2.2]{Sch15}, \cite[Sec.~11]{KW15},
\cite[Thm.~1.1]{BW15} using complex analytic techniques. Arithmetic methods  for understanding
jumping loci in characteristic zero are
developed in \cite[Thm.~1.1]{PR04} for coherent cohomology and~\cite[Thm.~1.5]{EK20} for $\ell$-adic sheaves.
\end{rmk}
We expect Theorem~\ref{thm.jumpint} to hold  for non-arithmetic
$\sF\in  D^b_c (X,\QQl )$ without the
conclusion on the $s_r$ being torsion points. This holds for
$\ch(F)=0$.

\subsection{Generic vanishing}\label{subsec:genvanint}

The study of generic vanishing was initiated  by Green--Lazarsfeld for the
cohomology  of line bundles  \cite[Thm.~1]{GL87}. As a third application we prove a
generic vanishing result for \'etale rank one $\QQl$-local systems. For this we formulate
a new abstract  approach to generic vanishing based on the Hard Lefschetz isomorphism and  the study of jumping
loci, see Section~\ref{subsec:genvan}. For simplicity  of exposition in the introduction, we confine the presentation
to the most
important special case of abelian varieties.

Let $X$ be an abelian variety of dimension $d$ over the algebraically closed field $F$ of
characteristic different from $\ell$. Let $\pi$ be the
$\ell$-adic completion of the abelian \'etale fundamental group $\pi_1^{\rm ab}(X)$ and
let the notation be as in Section~\ref{subsec:intjump}.

\begin{thm}\label{thm.genvan0}
 Assume that $\sF \in D^b_c( X, \QQl)   $ is  arithmetic and perverse. Then
    \[
\codim_{\sG( \QQl )}(\Sigma^i(\sF,0)) \ge |2i|
\]
for all $i\in\Z$.
\end{thm}

This is Corollary~\ref{cor:genvan}.

\begin{rmk}
  For $\ch(F)=0$ Theorem~\ref{thm.genvan0} is equivalent to \cite[Thm.~4.1]{Sch15} and~\cite[Thm.~1.3]{BSS18}, who do not
need the arithmeticity assumption, see also~\cite[Thm.~1.1]{KW15}.
For $\ch(F)>0$,  the  inequality    $$\codim(\Sigma^i(\sF,0)) >0 \ {\rm for} \ i\ne 0$$ is shown in \cite[Intro]{Wei16}.
\end{rmk}

Our proof of Theorem~\ref{thm.genvan0} relies on a study of a Galois tower
\[
\cdots \to X_{n+1} \to X_n \to \cdots \to X_0 =X
\]
which is
``isotropic'' with respect to the Weil pairing associated to a polarization $\eta$.
Here ${\rm Gal}(X_n/X) = (\Z_\ell/\ell^n\Z_\ell)^r $, $r\le d$. This tower gives rise to a
``tautological'' rank one \'etale $\mathfrak R$-local system $\sL_{\mathfrak R}$ on $X$.
Here the noetherian Jacobson ring $\mathfrak
R=\Z_\ell \llbracket \pi \rrbracket \otimes_{\Z_\ell} \QQl$ is defined in terms of the
completed group ring $\Z_\ell \llbracket \pi \rrbracket$, see Section~\ref{subsec:formbas}.

The idea for the proof of Theorem~\ref{thm.genvan0} is to use
on the one hand an
isomorphism property of the Lefschetz operator on finitely generated $\mathfrak R$-modules
\[
  \cup \eta^i\colon H^{-i}(X,\sF\otimes\sL_{\mathfrak R}) \to H^i(X,\sF\otimes\sL_{\mathfrak R}) .
\]
 On the other hand the choice of the tower leads to a
vanishing result for the Lefschetz operator $\cup \eta^i$ for $i\ge 1+d-r $, see Proposition~\ref{prop.keyvanish} and Lemma~\ref{lem:crucial}.

\subsection{Main Theorem} \label{ss:MainThm}

The three theorems presented in this introduction all rely on  the Main Theorem which does
not refer to any variety $X$. It describes  a Zariski closed locus of the space of
$\QQl$-points  of a multiplicative formal
Lie group under the condition that the locus is invariant under a specific type of linear automorphism.

We use the notation introduced in~\ref{subsec:intjump}, i.e.\
let $\pi$ be a finitely generated free $\Z_\ell$-module  with corresponding character
group $\sG(\QQl)=\Hom_{\rm cont} (\pi ,\QQl^\times )$. As before we endow $\sG(\QQl)$ with a Zariski
topology via the canonical  bijection
\[
 \Spm (\mathfrak R) \xrightarrow{\sim} \sG(\QQl) \quad J\mapsto \left(\pi\to (\mathfrak R/J)^\times
   =\QQl^\times \right)
\]
where $\mathfrak R=\Z_\ell \llbracket \pi \rrbracket \otimes_{\Z_\ell} \QQl$ receives the
tautological character $\pi\to \mathfrak R^\times$, see Section~\ref{subsec:formbas}. The following
is Theorem~\ref{main.thm}.

\begin{thm}[Main Theorem] \label{main.thm0}
   Let $\sigma$ be a $\Z_\ell$-linear automorphism of  $\pi$ such that $\sigma $ acts
   semi-simply on   $
\pi\otimes_{\Z_\ell} \Q_\ell$  and such that for a given complex embedding $\iota\colon
\QQl \hookrightarrow \C$, its eigenvalues $\alpha_i$ verify $|\iota (\alpha_i)|=|\iota (\alpha_j)|\neq 1$ for all $i,j$.
 Let
$S\subset  \sG(\QQl)$ be a Zariski closed subset. Then,
if $\sigma(S)= S$  we have
  \[
    S = \bigcup_{r\in I} s_r \sH_r( \QQl ) ,
  \]
  where $I$ is finite, $s_r\in \sG(\QQl)$ are torsion points and $\sH_r$ are formal
  Lie subgroups  of $\sG$.
\end{thm}

We briefly describe how we apply the Main Theorem.

The  idea of our proof of Theorem~\ref{thm:hl1} comes from~\cite{Dri01}.
We use a deformation space of rank one $\QQl$-local
systems on $X$  isomorphic to  $\sG(\QQl)$ as above. Inside of this deformation space we define the
bad locus to be the set of rank one local systems $\sL$ for which Theorem~\ref{thm:hl1}
fails to hold. The bad locus  is constructible in $\sG(\QQl)$ and we let $S$ be its Zariski
closure. Then $S$ is stabilized by a suitable Frobenius action $\sigma$ satisfying the
conditions of the Main Theorem. So it implies that $S$ is a union of torsion translated
Lie subgroups. Consequently, the torsion rank one local systems are dense in $S$. In view of Deligne's Hard
Lefschetz theorem the bad locus cannot contain a torsion local system, so it is empty.

For Theorem~\ref{thm.jumpint} we use that  the jumping loci, which  are closed subsets of $\sG(\QQl)$,
satisfy the assumption of the Main Theorem with respect to a suitable Frobenius action
$\sigma$.  The proof of Theorem~\ref{thm.genvan0} combines Theorem~\ref{thm:hl1} and
Theorem~\ref{thm.jump} as sketched at the end of Section~\ref{subsec:genvanint}.

\medskip

We now describe the idea of the proof of the Theorem~\ref{main.thm0}. We can assume that
$S$ is irreducible.       The conclusion
in particular implies that the torsion points are Zariski dense in $S$  (see
Lemma~\ref{lem:torsiondense}). So we first construct one torsion point on $S$.
To this aim, we replace  $S$ by the closed subset $\ell^n S$ of the group $\sG(\QQl)$ for some
$n\gg 0$ so as make sure that it cuts non-trivially
a small $\ell$-adic  neighborhood of $1\in \sG(\QQl)$ on which the $\ell$-adic logarithm map is an
isomorphism. This
enables one to transfer the problem  to a closed $\ell$-adic polydisc in the Lie algebra
 of $\sG$, i.e.\ to a  Tate
algebra on which $\sigma$ acts linearly.

Proposition~\ref{prop1} and~\ref{prop2}, proven using $\ell$-adic analysis,
show that in  our logarithmic coordinate chart $S$ contains the origin and moreover is
conical.
This implies that $\ell S\subset S$ by a density argument.  We can
then apply a theorem of de Jong~\cite[Prop.~1.2(1)]{deJ00} to finish the proof.
Alternatively,  what was our initial
proof, one can use the Weierstrass preparation theorem to find another torsion point on
$S$ contained in its regular locus,
in case $1\in S$ was singular. Then one  shows using simple $\ell$-adic analysis that if $S$
contains a torsion point in its regular locus  such that $S$ is conical around this
point then $S$ is linear. This is similar to the approach in~\cite[Section~4]{EK20}. As de Jong's argument   shortens our initial argument, we do
not give the details of it in this note.

\medskip

{\it Acknowledgements}:   Part of this work has been initiated while the first author was at
MSRI, then at IHES. We thank the two institutions for excellent working conditions.
We thank
Johan de Jong for kindly communicating to us his article~\cite{deJ00}, which allowed us to
abbreviate our original argument. We thank the referee for the friendly, sharp  and very helpful report.
 We deeply acknowledge the influence
of~\cite{Dri01} and of~\cite{BSS18} on our method.

\section{Tate algebras}

\noindent
Let $E$ be a finite extension of $\Q_\ell$ with residue field $k$ and uniformizer $\lambda
\in \sO_E$. We always fix an embedding $E\hookrightarrow \QQl$ into an algebraic closure
of $\Q_\ell$.
We let $|-|\colon E\to \R$ be the $\ell$-adic absolute value normalized by $|\ell|=1/\ell$.
Let $A=E\langle
T_1, \ldots , T_b \rangle$ be the Tate algebra, see~\cite[Sec.~3.1]{FvdP04}. Let $| - |$ be the Gauss norm on $A$, i.e.  $|g|={\rm sup}_n  |g^{(n)}|$ where $g^{(n)}$ is the coefficient of $\underline{T}^n, \  n\in \N^b$  in the expansion of $g$.
 We denote by $M$ the maximal ideal $(T_1, \ldots , T_b)\subset A$.

Let  $\sigma\in {\rm Mat}_b(\sO_E)$ be the diagonal matrix $\sigma= {\rm diag}(\alpha_1,
\ldots ,  \alpha_b )$. Then $\sigma$ induces an endomorphism of $A$ as an $E$-algebra by
the rule
$T_j\mapsto \sum_{i=1}^b \sigma_{i j} T_i$, which we
also denote by $\sigma$.

\begin{prop}\label{prop1}
Assume that $\underline \alpha^{ n}\ne 1$ for all
$n\in \N^b\setminus \{ 0\}$.
Let $I\subset A$ be an ideal with $\sigma(I)\subset I$ which is not
contained in the maximal ideal $ M$. Then $I=A$.
\end{prop}

\begin{proof} 
As $I$ does not lie in $M$, there is
a $g_\circ\in I$ with $g_\circ(0)=1$, which we fix for the rest of the proof.
The basic idea of the proof is simple: 
successively apply linear  expressions in $\sigma$ to $g_\circ$ in order to kill the
coefficients of degree $>0$ without changing the constant coefficient $1$. The problem is
to make such a sequence of elements of $I$ converge.

As any ideal in $A$ is closed \cite[Thm.~3.2.1]{FvdP04}, so is the subset
  \[
    \tilde I= \{ g\in I\,|\, g(0)=1, |g^{(n)}|\le |g_\circ^{(n)} | \}
  \]
of $A$.

  \begin{lem}\label{lem.prop1}
$\tilde I$ is compact.
    \end{lem}

    \begin{proof}
    As a Banach space, $A$ is isomorphic to $c_0$, the set of  sequences $g=(g^{(n)})_{n\in \N^b} $ in $E$
    with $|g^{(n)}|\to 0$ as $|n|\to \infty$.  It is endowed with the supremum norm $|g|={\rm sup}_n |g^{(n)}|$.
    One has an injective map
    $$\varphi\colon \prod_{n\in \N^b, g_\circ^{(n)}\neq 0} \sO_E \to J\subset c_0, \ (u^{(n)}) \mapsto  (u^{(n)}g_\circ ^{(n)}),$$
    defining $J$ as its image. As $\tilde I=J\cap I\cap \{(g^{(n)})_n, \ g^{(0)}=1\}$, $I$ is closed and the map $c_0\to E, \ (g^{(n)})\mapsto g^{(0)}$ is continuous, we just have to prove that $J$ is compact. 
The map  $\varphi$ is  continuous for the product topology on the left and the
   restriction of the topology of $A$ to $J$ on  the right.  Indeed, for $\rho>0$,  the inverse image
   \ml{}{ \varphi^{-1} (\{ |g|<\rho \})=  \\  \prod_{ n\in \N^b, 0< |g_\circ^{(n)}| < \rho}\sO_E\times
   \prod_{ n\in \N^b,  |g_\circ^{(n)}| \ge  \rho} \{u^{(n)}, |u^{(n)}|<  | g_\circ^{(n)}|^{-1}\rho\} \notag}
   is open
   as the index set      $\{n\in \N^b,  |g_\circ^{(n)}| \ge  \rho\}$ on the right is
   finite.

   As $\sO_E$ is compact, so is  $\prod_{n\in \N^b, g_\circ^{(n)}\neq 0} \sO_E$ by Tychonoff's theorem.  So its image $J$ by the continuous map  $\varphi$ is compact as well. This finishes the proof.
\end{proof}

    For $i\ge -1,$ set
    $$N_i= \{ n\in \N^b\setminus\{0 \}\, |\, |1-\underline \alpha^n| \ge |\lambda^i|  \}.$$
   This is an ascending chain of subsets of  $\N^b\setminus\{0 \} $  such that
   $$N_{i+1}\setminus N_i=\{ n\in \N^b\setminus\{0 \}\, |\, |1-\underline \alpha^n| = |\lambda^{i+1}|  \}.$$
   Because  $\alpha_i\in \sO$ and $\underline{\alpha}^n\neq 1$ for all $n\in \N^b\setminus \{0\}$, 
  one has
    $$N_{-1}=\varnothing,  \ \ \cup_i N_i=\N^b\setminus\{0 \}.$$
We define closed subsets \[\tilde I_i=
      \{ g\in \tilde I\,|\, g^{(n)}=0 \text{ for all } n\in N_i \}\subset I.
    \]
    In particular, $\tilde I_{-1}= \tilde I$ contains $g_\circ$, so it is non-empty. The
    $\tilde I_i$ form a descending chain of subsets of $\tilde I$.

    \begin{claim}\label{cl.prop2}
For all $i\ge 0$ the set $\tilde I_i$ is non-empty.
\end{claim}

\begin{proof}
It suffices to show that if $\tilde I_i$ is non-empty, so is $\tilde I_{i+1}$. In the
following we fix
$i\ge -1$.
We set $\sM=N_{i+1}\setminus N_i \subset N_{i+1}$, which we assume to be non-empty else the problem is solved.
We choose a linear order $m_1\prec m_2 \prec \ldots $  of $\sM=\{m_1, m_2,\ldots \}$. For
${\rm card}(\sM)\ge j\ge 1$ we define a closed subset
 \[
 \tilde I^{j}_i=\{ g\in \tilde I_i\,|\,
 g^{(m)}=0 \text{ for } m\prec  m_j, m\in \sM \} \subset \tilde I.
\]
Set $\tilde I^j_i=\tilde I_{i+1}$ for $j> {\rm card}(\sM)$.
 This is a  decreasing sequence of subsets of $\tilde I_i$ such that
$$\cap_{ j\ge 1 } \tilde I^{j}_i =\tilde I_{i+1}.$$
Thus we reduce the problem to showing that if $\tilde I^{j}_i \neq \varnothing$ for some $j\ge 1$,  then  $\tilde I^{j+1}_i \neq \varnothing$ for
 by Lemma~\ref{lem.prop1}   we then conclude
 $\tilde I_{i+1} \neq \varnothing.$

 We define a map $\Phi^m\colon  I\to  I$ by
\ga{}{  \Phi^m(g)= \frac{\sigma (g)- \underline{\alpha}^{ m }g}{1- \underline{\alpha}^{
      m}} \notag}
for $m\in \N^b \setminus\{ 0 \}$.
Then we obviously have $\Phi^{m}( g)^{(m)}=0, \ \Phi^{m}( g)(0)=1.$
In order to conclude the proof of the Claim~\ref{cl.prop2}, it remains to prove
Claim~\ref{cl.prop4}. Indeed, we then get $\Phi^{m_j} (\tilde I^j_i )\subset \tilde I^{j+1}_i$.
\end{proof}

\begin{claim}\label{cl.prop4}
$|\Phi^{m}( g)^{( n)} | \le |g^{( n)} | $ for $g\in \tilde I_i$, $n\in \N^b$ and $m\in \sM$.
\end{claim}
\begin{proof}
As $\Phi^m(g)^{(n)}=g^{(n)}=0$ for $n \in N_i$ we assume that $ n \notin N_i $. This implies
$|1-\underline\alpha^{n}| \le |1-\underline\alpha^{m} |$.
We have
$$
\Phi^m(g)^{( n)} = \frac{\underline\alpha^{ n} - \underline\alpha^{ m}}{1- \underline\alpha^{m}}  g^{( n)}. $$
So
\ml{}{
|\Phi^{m}(g)^{( n)} |=
\left|   \frac{\underline\alpha^{ n} - \underline\alpha^{m}}{1-
    \underline\alpha^{ m}}  \right| |g^{(n)}| \le \\ \frac{\max ( |1- \underline\alpha^{
    n}|, |1 - \underline\alpha^{ m}|)}{|1- \underline\alpha^{ m}|}|g^{(n)}|    \le |g^{(n)}|. \notag
}
This finishes the proof.
\end{proof}

Lemma~\ref{lem.prop1}  and Claim~\ref{cl.prop2}    imply that $\cap_{i} \tilde I_i$  is
non-empty, say it contains $h$. Then $h^{(n)}=0$ for all $n\in \N^b\setminus\{ 0 \}$, so
 $h \in I$ is a constant and hence $1\in I$. 
This  finishes the proof of Proposition~\ref{prop1}.
\end{proof}

In the next proposition we combine Proposition~\ref{prop1} with a ``weight'' argument in
order to deduce that $I$ is homogeneous under suitable assumptions.
We fix a complex embedding  $\iota\colon \QQl \hookrightarrow \C$.

\begin{prop}\label{prop2} Let $\sigma\in {\rm GL}_b(\sO_E)$ be the diagonal matrix $\sigma
  = {\rm diag}(\alpha_1,
\ldots ,  \alpha_b )$.
 We assume  that  $|\iota
 (\alpha_i)|=|\iota (\alpha_j)| \neq 1 $ for all $1\le i,j \le b$.  Let
 $I\subset A$ be a radical  ideal
with $\sigma (I)=I$.  Then $I$ is homogeneous.
\end{prop}

\begin{proof}
To start with we observe that as a consequence of our assumption on the eigenvalues
$\alpha_i$ we obtain the implication
\begin{equation}\label{eq.prophomo}
 \underline \alpha^{n} =1 \quad \Rightarrow \quad n_1 + \cdots + n_b=0
\end{equation}
 for $n\in \Z^b$.

  As $I$ is the intersection of finitely many minimal prime ideals containing
  it and as $\sigma$ permutes these prime ideals, we can assume without loss of generality that
  $I$ itself is a prime ideal.
The assumptions on the eigenvalues $\alpha_i$ in Proposition~\ref{prop1} are satisfied
by~\eqref{eq.prophomo}, so  we see that $I\subset M=(T_1, \ldots , T_b)$.

One has to check that the homogeneous components of an element $g\in I$ are in $I$. Denote by
$\hat A$ and $\hat I$ the completion at the maximal ideal $M=(T_1, \ldots , T_b)$. As the map
$A/I\to \hat A/\hat I$ is injective, it suffices to show that the
homogeneous components of $g$ are in $\hat I$. This is equivalent to saying that the homogeneous components of
$g$ are in the ideal $ I+M^n/ M^n\subset  A/ M^n$ for all $n>0$, i.e.\ that  $ I+M^n/ M^n$
is graded by degree.

Let $D\hookrightarrow {\rm GL}_{b,E}$ be the smallest linear algebraic subgroup (over $E$) containing $\sigma^m$
for all $m\in \Z$. We can determine the diagonalizable group $D$ a follows. Let $T\hookrightarrow {\rm GL}_{b,E}$ be the
standard maximal torus. Then $D$ is the intersection of the kernels of all characters
$\chi_n\colon T\to \mathbb G_m$ with $\chi_n(\underline \alpha)=1$. Here $n\in\Z^b$ and $\chi_n(\underline
\alpha) =\underline \alpha^n$. By~\eqref{eq.prophomo} we see that those $\chi_n$ are
trivial on the diagonal torus $\mathbb G_m$, so $ D$ contains the diagonal $\mathbb G_m$.

But $D$   acts on $A/M^n$ for any $n>0$ and this action preserves  $I+M^n/M^n$. The weight
decomposition with respect to the action of the diagonal $\mathbb G_m\subset T$ implies
that $I+M^n/M^n$ is graded by degree.
  \end{proof}

\section{Multiplicative formal groups} \label{sec:formgr}

\noindent
In this section we prove our main theorem on multiplicative formal Lie groups, Theorem~\ref{main.thm}.

\subsection{Basics}\label{subsec:formbas}

We recall some basic results on multiplicative formal groups and we define a Zariski
topology on its group of $\QQl$-points.

Let $E$ be a finite extension of $\Q_\ell$ and fix an embedding $E\hookrightarrow \QQl$.
For an abelian pro-finite group $\pi$ we let
\[
R=\sO_E \llbracket \pi \rrbracket = \lim_\Lambda \sO_E [\Lambda] = \lim_{m, \Lambda} (\sO_E/ (\lambda^m)) [\Lambda]
\]
where $\Lambda$ runs through the system of finite quotients
of $\pi$ and $(\lambda)\subset \sO_E$ is the maximal ideal. The quotient $(\sO_E/ (\lambda^m)) [\Lambda]$ is finite and is endowed with the  discrete topology.
Then $\sO_E \llbracket \pi \rrbracket$  is endowed with the limit topology.

  On $\sO_E \llbracket \pi
\rrbracket$ we consider the usual completed Hopf algebra structure over $\sO_E$, for example  the
multiplication   is given by
\[
\sO_E \llbracket \pi \rrbracket\, \widehat \otimes_{\sO_E} \sO_E \llbracket \pi \rrbracket
\to \sO_E \llbracket \pi \rrbracket \quad [e_1]\otimes [e_2]\mapsto [e_1+ e_2].
\]
For another pro-finite group $\pi'$ the
continuous  Hopf algebra homomorphisms
 $\sO_E \llbracket \pi
\rrbracket\to \sO_E \llbracket \pi' \rrbracket$ are in bijection with the continuous
homomorphisms $\pi\to \pi'$.  One can show this by reducing to $\pi$ and $\pi'$ finite and then applying
Cartier duality~\cite[Exp.~VIIB, 2.2.2 Prop.]{SGA3}.  We denote the prime ideal generated by $1-[e]$
for all $e\in \pi$ by $M$, i.e.\ $M$ is the kernel of the counit $  \sO_E \llbracket \pi
\rrbracket \to \sO_E$, $[e]\mapsto 1$ for $e \in \pi$.

In the following we assume that $\pi$ is a finitely generated free $\Z_\ell$-module of rank $b$. In
this situation one says that the associated formal
group $\sG=\Spf (\sO_E \llbracket \pi \rrbracket)$ is a {\it multiplicative} $b$-{\it
  dimensional formal Lie group over} $\sO_E$.
A closed formal subgroup $\sH \hookrightarrow \sG$ is called a {\it formal Lie subgroup}
if it  corresponds to a quotient morphism of $\sO_E$-algebras $\sO_E \llbracket \pi \rrbracket \to
\sO_E \llbracket \pi/\pi' \rrbracket$, where $\pi/\pi'$ is a torsion free quotient $\Z_\ell$-module of $\pi$.

Once we choose a $\Z_\ell$-basis $e_1,\ldots , e_b$ of  $\pi$ we  obtain an  isomorphism
\[
\sO_E\llbracket \pi \rrbracket \cong
\sO_E \llbracket X_1,\ldots , X_b \rrbracket
\]
defined by \[[e_i] \mapsto 1+X_i\]
and the comultiplication becomes
\ml{}{ \sO_E \llbracket X_1,\ldots , X_b \rrbracket\to \sO_E\llbracket Y_1,\ldots , Y_b, Y'_1,\ldots, Y'_b \rrbracket \\
X_i\mapsto Y_i+Y'_i+ Y_i Y'_i. \notag}

 We  identify $\sG(\QQl)=\Hom_{\sO_E}( \sO_E \llbracket \pi \rrbracket ,\QQl)$ with the
 group of continuous homomorphisms $\Hom_{\rm cont}(\pi, \QQl^\times)$.
Recall that by~\cite[Prop.A.2.2.3]{GL96},  $\sG(\QQl )$ can also be
identified with the maximal spectrum $\Spm(\mathfrak R)$ of $\mathfrak R = R
\otimes_{\sO_E} \QQl$. Furthermore, $\mathfrak R$ is a noetherian
Jacobson ring of dimension $b$ ({\it loc.cit.}).

\medskip

As a maximal spectrum, $\sG(\overline \Q_\ell)$  is endowed with a {\it Zariski topology}, which
is  the topology on $\sG(\overline \Q_\ell)$ we use in the sequel.
\begin{lem} \label{lem:torsiondense}
The subset of torsion points of $\sG(\overline \Q_\ell)$ is dense.
\end{lem}

\begin{proof} We have to show that  if $g\in \mathfrak R$  vanishes on all torsion points
  then $g=0$. Without loss of generality $g\in \sO_E \llbracket \pi \rrbracket$.
Let $[\ell^n ]: \sO_E \llbracket \pi \rrbracket \to \sO_E \llbracket \pi \rrbracket   $ be the
morphism induced by $\ell^n$-multiplication on $\pi$.
Let $J_n$ be the ideal generated by $[\ell^n](M)$. 
For all $n>0$ the ring
\begin{equation}\label{eq:lem3.1..}
\sO_E \llbracket \pi
\rrbracket/J_n= \sO_E [ \pi/\ell^n \pi ]
\end{equation}
is flat over $\sO_E$ and  its tensor product with $E$ is reduced.
As the finite group $\Spm(\sO_E [ \pi/\ell^n \pi ]\otimes_{\sO_E} \QQl)$ identifies with the
$\ell^n$-torsion points of $\sG(\QQl)$, we see that the image of $g$ in the
rings~\eqref{eq:lem3.1..} vanishes for all $n>0$. As the limit over $n$ of the
rings~\eqref{eq:lem3.1..} is $\sO_E \llbracket \pi
\rrbracket$ by definition, we get $g=0$.
\end{proof}

\subsection{Exponential map}

We recall  some well-known facts on the $\ell$-adic exponential map,
see~\cite[Ch.~12]{Cas86}.
 We consider the $\ell$-adic absolute value $|-|$ on $\QQl$ with $|\ell| =1/\ell$.
For $\rho\in (0,1)\cap |\QQl^\times |$  we have extensions of rings
\[
\sO_E \llbracket X_1,\ldots , X_b \rrbracket \otimes_{\sO_E} E \subset E\langle  X_1,\ldots ,
X_b  \rangle_{\rho}
\]
where the Tate ring on the right is the completion of the polynomial ring with respect to the $\rho$-Gauss norm.
Recall that the  $\rho$-{\it Gauss norm} of $g=\sum_{n} g^{(n)} \underline X^n$ is defined by
$|g|=\sup_n | g^{(n)}| \rho^n $.

Under the additional assumption $\rho <\ell^{-1/(\ell-1)}$,
we have an exponential isomorphism
\[
 E\langle  X_1,\ldots , X_b  \rangle_{\rho} \xrightarrow{\sim}  E\langle  T_1,\ldots , T_b  \rangle_{\rho} , \ X_i \mapsto \exp (T_i)-1.
\]

Let $B(\rho)\subset \QQl^b$ be the polydisc consisting of points with maximum norm $\le \rho$, where
$\rho \in |\QQl^\times |$. As $B(\rho)$ can be identified with the maximal spectrum of \[
\QQl \langle  T_1,\ldots , T_b  \rangle_{\rho} =\colim_E E\langle T_1, \ldots ,
T_b\rangle_\rho , \] we can endow it with a Zariski
topology.

Independently of the choice of coordinates one can identify $B(\rho)$ with $\Hom_{\ZZl}
(\pi \otimes_{\Z_\ell} \beta \ZZl , \ZZl)$, where $\beta \in \QQl$ is such that $|\beta| =1/\rho $.
So more generally for any finitely generated  $\ZZl$-submodule $\pi'\subset
\pi\otimes_{\Z_\ell} \QQl$ containing $\pi$,  there exists a closed polydisc $ B_{\pi'}
\subset B(1) $.

Similarly, we let $\sG(\QQl)(\rho)$ be the subgroup of  $\sG(\QQl)$
  which consists of the continuous homomorphisms
$\chi\colon \pi\to \QQl^\times $  with $|\chi(e)-1|\le \rho $ for all $e \in \pi$.
We  then obtain an injective group homomorphism
\[
  \exp_\rho\colon B(\rho) \to \sG(\QQl) \quad \text{ for } \rho <\ell^{-1/(\ell-1)}
\]
which is continuous with respect to the Zariski topology and the image of which  is
$\sG(\QQl)(\rho)$.


\begin{lem}\label{lem.denserho} Let $S\subset \sG(\QQl )$
be an irreducible closed subset  and $\rho\in (0,1)\cap |
 \QQl^\times |$ be such that
the subset $S \cap \sG (\QQl)(\rho)$ is non-empty. Then  $S \cap \sG (\QQl)(\rho)$ is
dense in  $S$.

\end{lem}

\begin{proof}
After replacing $E$ by a finite extension we can assume that $S$ is given by an integral quotient
ring $A$ of $\sO_E \llbracket \pi \rrbracket$ which is flat over $\sO_E$.
The map $A\to A\otimes_{\sO_E \llbracket \pi \rrbracket}  \QQl \langle  X_1,\ldots ,
X_b  \rangle_{\rho}   $ is injective since it is flat and non-zero. The codomain of this map is a reduced
Jacobson ring~\cite[Prop.~2.2]{EK20}  and its maximal ideals correspond to $ S \cap \sG (\QQl)(\rho) $.
As the closure of  $S \cap \sG (\QQl)(\rho)$ corresponds to the intersection of these
maximal ideals inside $A$, which is the zero ideal, we deduce Lemma~\ref{lem.denserho}.
\end{proof}

\subsection{Main Theorem}

This subsection contains the technically central result of our note.
Let $E$ be a finite extension of $\Q_\ell$ together with a fixed embedding
$E\hookrightarrow \QQl$ and let $\sG=\Spf( \sO_E \llbracket \pi \rrbracket )$ be a multiplicative
$b$-dimensional formal Lie group over  $\sO_E$. Let $\sigma\colon \pi\to \pi$ be an automorphism such
that the $\Q_\ell$-linear map $\sigma\colon \pi\otimes_{\Z_\ell} \Q_\ell\to \pi\otimes_{\Z_\ell}
\Q_\ell$ is {\it semi-simple} with eigenvalues $\alpha_1, \ldots , \alpha_b\in \QQl$.
We also denote by $\sigma\colon\sG \to \sG$ the corresponding automorphism of formal groups.
Recall that we fix a complex embedding $\iota\colon \overline \Q_\ell \hookrightarrow \C$.

We define quasi-linearity following de Jong~\cite[Def.~1.1]{deJ00}.

\begin{defn}\label{def:ql}
  A closed subset $S\subset \sG(\QQl)$  is called quasi-linear if it can be written in
   the form
  \[
    S=\bigcup_{r\in I} s_r   \sH_r(\QQl),
  \]
 where
   $I$ is finite, the elements
 $s_r\in \sG(\overline \Q_\ell )$ are torsion and the  $\sH_r$ are  formal Lie
 subgroups of $\sG$.
\end{defn}

\begin{thm}\label{main.thm}
 Assume that for all $i,j\in \{1,\ldots , b \}$ we have $|\iota (\alpha_i)|=|\iota (\alpha_j)| \ne 1$.
 Let $S\subset \sG(\QQl )$ be a Zariski closed subset with $\sigma(S)=S$. Then
 $S$ is quasi-linear.
\end{thm}

\begin{proof} 
  We can
  assume that $S$ is non-empty and irreducible. We fix $\rho\in (0,
  \ell^{-1/(\ell-1)})\cap |\QQl^\times |$.
There exist $\ZZl$-linear independent eigenvectors
$e'_1,\ldots , e'_b\in \pi\otimes_{\Z_\ell} \ZZl$   of $\sigma$ such that $
\pi $ is
contained in $\pi'=\ZZl e'_1 + \cdots +\ZZl e'_b$. We also fix an integer $w>0$ with
$\ell^w\pi'\subset \pi\otimes_{\Z_\ell} \ZZl$.
  
  There exists $n>0$ such that $[\ell^n]( S)\cap \sG(\rho/\ell^w)$ is non-empty. As $[\ell^n]\colon
  \sO_E \llbracket \pi \rrbracket \to  \sO_E \llbracket \pi \rrbracket $ is a finite,
  faithfully flat ring homomorphism, we deduce that $[\ell^n]( S)$ is closed in $\sG(\QQl)$. So
  $\exp_\rho^{-1}( [\ell^n](S) )\cap B_{\pi'}$ is closed and non-empty in the polydisc
  $B_{\pi'}$. The choice of $e'_1,  \ldots , e'_b$ above allows us to identify $B_{\pi'}$
  with the maximal spectrum of the Tate algebra $\QQl\langle T'_1 , \ldots , T'_b
  \rangle$. So   after replacing $E$ by a
  finite extension the Zariski closed subset $\exp_\rho^{-1}( [\ell^n](S) )\cap B_{\pi'}$
  of $ B_{\pi'}$ corresponds to a radical ideal $I\subset  E\langle  T'_1,\ldots , T'_b
  \rangle $.


  We can apply Proposition~\ref{prop2} in order to see that
  $I$ is homogeneous.
  Consequently,  $\exp_\rho^{-1}( [\ell^n](S) )\cap B_{\pi'}$ and therefore also its subset  $\exp_{\rho/\ell^w}^{-1}( [\ell^n](S) )$ is stabilized by
  the homothety $\ell$,
  which is equivalent to the fact that $[\ell^n](S)\cap
  \sG(\rho/\ell^w) $ is stabilized by $[\ell]$.
  As $[\ell^m](S)\cap
  \sG(\rho/\ell^w) $ is Zariski dense in $[\ell^m](S)$ for all $m\ge n$ by
  Lemma~\ref{lem.denserho}, we deduce that $[\ell]$ also stabilizes $[\ell^n](S)$. By a
  result of de Jong~\cite[Prop.~1.2(1)]{deJ00} this implies the theorem.
\end{proof}

\section{Generalized Fourier-Mellin transform}

\noindent
In this section we  consider a generalization of $\ell$-adic cohomology which for tori is called Mellin transform in~\cite[ Prop.~3.1.3]{GL96}
and which for complex abelian varieties is a completion of the Fourier-Mellin transform
in~\cite[Section~1]{BSS18}. The use of this ``Fourier-Mellin transform'' is limited by the
fact that we do not know any sort of inversion formula at the moment.
The only really new
result in this section is Proposition~\ref{prop.keyvanish} which provides a simple direct
approach to generic vanishing.

\subsection{Definition and basic properties}\label{subsec:meldef}

Let $X$ be a separated, connected scheme of finite type over the algebraically closed field
$F$. All cohomology groups we consider will be with respect to the \'etale topology. {\it In the sequel, a tensor
product involving a derived object (module or sheaf)   means a derived tensor product.}

Let  $E$ be a finite extension of $\Q_\ell$ with a fixed embedding $E\hookrightarrow \QQl$.
Let $R$ be a complete noetherian local $\sO_E$-algebra with finite residue field and maximal
ideal $\mathfrak m$.  Then $\mathfrak
R=R\otimes_{\sO_E} \QQl$ is a noetherian Jacobson ring by~\cite[Prop.A.2.2.3]{GL96}.

Let $\sF$ be in $D^b_c(X, \sO_E)$ and let $\rho\colon\pi_1(X)\to
R^\times$ be a continuous character.
We define $\mathfrak m$-adic \'etale cohomology as
\[
R \Gamma(X,\sF \otimes_{\sO_E} \sL_{R} )  := R \lim_n  R \Gamma(X,\sF \otimes_{\sO_E}
\sL_{R/\mathfrak m^n} ),
\]
where  $\sL_{R/\mathfrak m^n}$ is the \'etale local system on $X$ associated to the finite character
$\pi_1(X)\xrightarrow{\rho}  R^\times  \to (R/\mathfrak m^n)^\times$.
We denote the cohomology of this complex by $H^i(X, \sF \otimes_{\sO_E} \sL_{R})$.
The corresponding cohomology with compact support is defined in the usual way.
Let $\omega_X \in D^b_c(X,\sO_E )$ be the dualizing complex $f^!(\sO_E )$, where $f\colon X\to \Spec(F)$ is the canonical map.

We collect some properties of this $\mathfrak m$-adic cohomology, which follow from \cite{Eke90}.

\begin{prop}\label{prop:propertfm}  \mbox{}
  \begin{itemize}
  \item[(1)]{\rm [Finiteness]} The complex
    $ R \Gamma(X,\sF \otimes_{\sO_E} \sL_{R} ) $ has bounded, coherent cohomology groups,
    i.e.\ it is in $ D^b_\coh (R)$.
  \item[(2)]{\rm [Base change]} For any quotient ring $R'$ of $R$   we have  a base change isomorphism
    \[
  R \Gamma(X,\sF \otimes_{\sO_E} \sL_{R} )  \otimes_{R} R'  \xrightarrow{\simeq}   R
  \Gamma(X,\sF \otimes_{\sO_E} \sL_{R'} )\in D^b_{\rm coh}(R' ).
\]
\item[(3)]{\rm [Limit property]} We have an isomorphism of $R$-modules
  \[
H^i(X, \sF \otimes_{\sO_E} \sL_{R}  ) \xrightarrow{\simeq} \lim_n H^i(X, \sF \otimes_{\sO_E} \sL_{ R/ \mathfrak m^n}).
  \]
\item[(4)]{\rm [Duality]} There is a canonical isomorphism
  \[
    R \Hom_R ( R \Gamma_c(X,\sF \otimes_{\sO_E} \sL_{R} ),R ) \simeq  R
    \Gamma(X,\sF^\vee \otimes_{\sO_E} \sL^\vee_{R} ) 
  \]
  in $D^b_{\rm coh}(R)$,
  where $\sF^\vee = R\Hom(\sF,
    \omega_X)$
  and where $\sL^\vee_R$ is the local system associated to the dual character $\rho^{-1}$.
\end{itemize}
\end{prop}

Part (1), (2) and (4) follow from~\cite[Thm.~6.3, Thm.~7.2]{Eke90},  part (3) follows from the fact that $\lim_n^1
H^i(X,R/ \mathfrak m^n)$ vanishes for all $i\in \Z$ as these $R$-modules  are  artinian.
Note that Ekedahl assumes that $R$ has finite global dimension, which is sufficient for
our application. It is however not difficult to show the general case.

Note that for $\sF\in  D^b_c(X,\QQl)$ we get a corresponding complex
$$ R \Gamma(X,\sF \otimes
\sL_{\mathfrak R} )  \
{\rm in} \   D^b_{\rm coh}(\mathfrak R ),$$
associated to the character  $\pi_1(X)\xrightarrow{\rho}  R^\times  \to \frak{R}^\times$.
\medskip

  One way in which this \'etale $\mathfrak m$-adic cohomology is useful is the following
  isomorphism criterion for a cup-product.
 Let $\sF $ and $\sK$ be in $D^b_c(X,\QQl)$ and assume that $R$ is an integral  domain  in
 which $\ell$ does not vanish.

  \begin{lem}\label{lem.geniso}
    For  $\xi\in H^j(X,\sK)$ and $i\in \Z$ the following are equivalent:
    \begin{itemize}
    \item[(1)] The cup-product
      \[
     H^i(X, \sF \otimes \sL_{\mathfrak R}  )    \xrightarrow{\cup\xi}  H^{i+j}(X, \sF \otimes
     \sK \otimes \sL_{\mathfrak R}  )
      \]
  is an
  isomorphism (resp.\ not an isomorphism) after tensoring with $\Frac (\mathfrak R )$.
\item[(2)] $H^i(X,\sF \otimes \sL_s )  \xrightarrow{\cup\xi} H^{i+j}(X,\sF\otimes \sK\otimes
    \sL_s)$ is an isomorphism (resp.\ not an isomorphism) for all $s\in U$, where
    $U\subset  \Spm (\mathfrak R )$ is a dense
    open subset.
  \end{itemize}
\end{lem}

Here $\sL_s$ for $s\in \Spm(\mathfrak R )$ is the local system on $X$ corresponding to the
character $\pi_1(X)\to \mathfrak R^\times \to k(s)^\times$, which is given by reduction modulo the
maximal ideal associated to $s$. Note that the residue field $k(s)$ is equal to $\QQl$,
see Section~\ref{subsec:formbas}.

  \begin{proof}
It suffices to prove (1) $\Rightarrow $ (2).
 By Proposition~\ref{prop:propertfm}(1) there exists
a dense open subset $U\subset \Spm(\sR)$ such that the following $\mathfrak R$-modules are
flat over $U$:  
\ml{}{
   H^*(X, \sF \otimes \sL_{\mathfrak R}  )
   ,\   H^*(X, \sF \otimes \sK \otimes  \sL_{\mathfrak R}  )\text{ and}  \notag \\ 
   \coker ( H^*(X, \sF \otimes \sL_{\mathfrak R}  ) \xrightarrow{\cup\xi}   H^*(X, \sF \otimes \sK \otimes
  \sL_{\mathfrak R} ) ) .
  \notag}
Note that then also the kernel of $\cup\xi$ is flat over $U$.
 Then the
conclusion follows from Proposition~\ref{prop:propertfm}(2) and the Tor-spectral sequence.
  \end{proof}

Combining  Lemma~\ref{lem.geniso}  with \cite[Prop.~0.9.2.3]{EGAIII}   we obtain:

  \begin{cor} \label{cor:cons}
The set of $s\in \Spm (\mathfrak R )$ with the property that \[H^i(X,\sF \otimes \sL_s )  \xrightarrow{\cup\xi} H^{i+j}(X,\sF\otimes \sK\otimes
    \sL_s)\] is an isomorphism (resp. not an isomorphism)  is constructible.
    \end{cor}

\smallskip

Now we consider a special ring $R$. Let $\pi$ be a torsion free $\ell$-adic quotient of
$\pi_1^{\rm ab}(X )$ and let $R$ be the completed group ring
\[
R = \sO_E \llbracket \pi \rrbracket = \lim_n \sO_E [\pi/ \ell^n \pi ].
\]
as in Section~\ref{sec:formgr}.
We let $\rho\colon \pi_1(X) \to R^\times$ be the canonical character $e\mapsto [e ]$.
Set $\frak{R}=R \otimes_{\sO_E} \QQl$.

\begin{defn}\label{def.forme}
  The integral Fourier-Mellin transform of $\sF\in D^b_c(X, \sO_E )$ is defined as
  \[
\FM_\pi(X,\sF) =   R \Gamma(X,\sF \otimes_{\sO_E} \sL_{R} )\in D^b_\coh(R).
\]
Up to isogeny we get an induced Fourier-Mellin transform $$\sFM_\pi(X,\sF) \in
D^b_\coh(\mathfrak R)$$ for $\sF\in D^b_c(X,\QQl )$.
The corresponding cohomology modules are denoted by  $\FM^i_\pi(X,\sF)$ resp.\  $\sFM^i_\pi(X,\sF)$.
\end{defn}

In case the group $\pi$ is clear from the context we omit it in the notation.

\begin{rmk}\label{rmk:fmdual}
 For $ R=\sO_E \llbracket \pi \rrbracket  $ and  for an $R$-module (sheaf) $\sM$ we denote by $\sM'$ the same
abelian group (sheaf) with the  $R$-module structure twisted by the automorphism $[-1]\colon R \xrightarrow{\sim}
R$, $[-1]([e]) = [-e]$   for $e\in \pi$.
Then the dual $R$-module sheaf $\sL_R^\vee$ is isomorphic
to $\sL'_R$.
\end{rmk}

\subsection{A vanishing result}

Again we fix a  finitely generated free $\Z_\ell$-module quotient $\pi$ of $\pi_1^{\rm ab}(X)$  and we set
$R = \sO_E\llbracket \pi \rrbracket$.
We have an induced tower of Galois coverings of $X$
\[
\cdots \to X_{n+1} \to X_n \to \cdots \to X_0=X
\]
with $\Gal(X_n/X)=\pi/\ell^n \pi$. We denote this tower by $X_\infty$ and we use the
notation
\[
  H^j(X_\infty , \sK )=\colim_n H^j(X_n , \sK )\]
for $\sK\in D^b_c(X,\sO_E)$.

The following vanishing proposition is our key new technical result which allows us to
obtain a short proof of the generic vanishing theorem. The analog for complex analytic varieties could
be used to give an alternative direct proof of ~\cite[Thm.~1.3]{BSS18},
\cite[Cor.~7.5]{Sch15},~\cite[Thm.~1.1]{KW15}.

\begin{prop}\label{prop.keyvanish}
Let   $\sF$  and $\sK$ be in $ D^b_c(X,\sO_E)$.
If $\xi\in H^j(X,\sK)$ becomes divisible in $H^j(X_\infty,\sK)$,  then the cup-product map
\[
\FM^i(X,\sF) \xrightarrow{\cup\xi} \FM^{i+j}(X,\sF\otimes_{\sO_E} \sK)
\]
vanishes for all $i\in \Z$.
\end{prop}

\begin{proof}
 Recall that   $R=\sO_E \llbracket \pi \rrbracket  $ and that this ring is identified with  $\sO_E \llbracket X_1,
  \ldots , X_b \rrbracket $ by sending $[e_i]$ to $1+X_i$, see Section~\ref{subsec:formbas}.
  The key observation is that we have an isomorphism of pro-rings
  \begin{equation}\label{eq.prosy}
\{ R / \mathfrak m^n \}_n \simeq   \{\sO_E/\ell^m \sO_E [\pi / \ell^n \pi]  \}_{m,n}.
\end{equation}
One easily sees the two isomorphisms of pro-rings
\begin{equation}\label{eq.proiso1}
\{ R / \mathfrak m^n \}_n \simeq \{ R /(\ell^m R +  X_1^n R +
\cdots + X_b^n R)  \}_{m,n}
\end{equation}
and
\begin{equation}\label{eq.proiso2}
\{ \sO_E/\ell^m [\pi / \ell^n \pi] \}_{m,n}  = \{  R/(\ell^m R +   ((X_1+1)^{\ell^n}-1)
R +  \cdots)  \}_{m,n}  .
\end{equation}
So to prove the isomorphism of pro-rings ~\eqref{eq.prosy} we have to show that for fixed $m>0$ the right sides of~\eqref{eq.proiso1} and
of~\eqref{eq.proiso2} are isomorphic as pro-systems in $n$.
As the ring on the right side of~\eqref{eq.proiso2} is artinian,  $X_i$ is nilpotent in
it, which shows one
direction.

Conversely, it suffices  to show that $ (X_i+1)^{\ell^r}-1$
vanishes in the ring on the right side of~\eqref{eq.proiso1}  for $r\gg 0$
depending on $m$ and $n$.
We verify this by observing that  in the ring $ R/\ell^m $ for any $n\ge m>0 $  we
have $X_i^{\ell^{n-m+1}} | (X_i+1)^{\ell^n}-1$. In order to verify this divisibility one has to show that
all coefficients of the integral polynomial $(X_i+1)^{\ell^n}-1 = \sum_{r=1}^{\ell^n}
\binom{\ell^n}{r} X_i^r $ are divisible by
$\ell^m$ in degrees $< \ell^{n-m+1}$.  This follows from  ${\rm ord}_\ell\left( \binom{\ell^n}{r} \right) =
n-{\rm ord}_\ell(r) $ for $r>0$. 

From Proposition~\ref{prop:propertfm}(3) and from the isomorphism~\eqref{eq.prosy} we deduce that 
\ga{}{ \FM^i(X,\sF)= {\rm lim}_{m,n}   H^i(X, \sF\otimes_{\sO_E} \sL_{\sO_E/\ell^m  [\pi / \ell^n \pi] } ). \notag}
We have
\begin{equation}\label{eq.cohocov}
  H^i(X, \sF\otimes_{\sO_E} \sL_{\sO_E/\ell^m  [\pi / \ell^n \pi] } ) \cong
  H^i(X_n, p_n^*(\sF)\otimes_{\sO_E}\sO_E/\ell^m )
\end{equation}
and similarly for $\FM^{i+j}(X,\sF\otimes_{\sO_E} \sK)$, where $p_{n}\colon X_n
 \to X$ is the canonical finite \'etale map .
For the
isomorphism~\eqref{eq.cohocov}  one uses that
 $p_{n\, *} \,\sO_E /\ell^m \cong  \sL_{\sO_E/\ell^m  [\pi / \ell^n \pi] }$ and the projection formula.
  Here  $ \sO_E /\ell^m$ is the constant sheaf on $X_n$.

Our assumption on $\xi$ says that if  $n$ is large enough, depending on $m$, then $\xi$ is
$\ell^m$-divisible in $H^j(X_n,\sK)$. In this situation  the
cup-product
\begin{equation}\label{eq.provan}
 H^i(X_n, p_n^*(\sF)\otimes_{\sO_E}\sO_E/\ell^m ) \xrightarrow{\cup\xi}  H^{i+j}(X_n, p_n^*(\sF)\otimes_{\sO_E}\sK\otimes_{\sO_E}\sO_E/\ell^m )
\end{equation}
vanishes. So taking the limit over $m$ and $n$ in~\eqref{eq.provan} and using the
projection formula we finish the proof of Proposition~\ref{prop.keyvanish}.
  \end{proof}

\section{Hard Lefschetz theorem} \label{sec:HL}

\subsection{Formulation of the theorems}

Let $X$ be a separated scheme of finite type  over an algebraically closed field $F$ of
characteristic different from $\ell$.

\begin{defn} \label{defn:arithmperv} A complex
$\sF\in D^b_c(X , \QQl )$ is called arithmetic, if there exists a  finitely
generated field 
$F_0\subset F$ such that
\begin{itemize}
 \item[(1)]
   $X$ descends to a separated scheme of finite type $X_0/F_0$;
\item[(2)]
   $\sF$
lies in the full subcategory
\[
D^b_c(X_0\otimes_{F_0}\overline F_0 , \QQl ) \hookrightarrow D^b_c(X , \QQl ),
\]
where $ F_0\subset  \overline{F_0}$ is the algebraic closure of $F_0$ in $F$;
\item[(3)]
  For each element  $\sigma\in
 {\rm Gal}(\overline F_0 /F_0)$, one has
\[ \sigma(\sF)\simeq \sF \in D^b_c(X_0\otimes_{F_0}\overline F_0 , \QQl ) .\]
\end{itemize}
\end{defn}
 In particular  if
$F=\overline \F_p$ then $\sF$ is arithmetic if and only if it is stabilized by a non-trivial
power of the Frobenius.
\begin{rmk} \label{rmk:arithm}
The notion of arithmeticity in Definition~\ref{defn:arithmperv} is the same as the one in   in~\cite[Defn.~1.4]{EK20},
replacing $\C$ by $F$.
  In
 \cite[p.~163]{BBD82}  {\it semi-simple}  $\sF\in D^b_c(X, \overline \Q_\ell)$  {\it of
   geometric origin} are defined (over $F=\C$, but this is irrelevant for the discussion) and
one easily checks that these $\sF$ are arithmetic.

In fact wherever we impose the arithmeticity condition in this note a slightly weaker
condition would be sufficient, which is however quite technical to formulate precisely.
 Recall  that the condition  labeled (P) in  \cite[Lem.~6.2.6]{BBD82} says that after a
  spreading of $X,$ a suitable  specialization of $\sF$
 to $\sF_{\bar s}$ is fixed by a Frobenius, where $\bar s$ is a sufficiently generic
 $\overline \F_p$-point.
 The specialization depends on the choice of a strictly henselian discrete valuation ring $V$
 with $ V \subset F$ such that the closed point of $\Spec(V)$ is $\bar s$.
 In our use of the arithmeticity condition, all we shall need is precisely this {\it
   invariance of $\sF$ under one generic Frobenius action}.
\end{rmk}

\smallskip

The Lefschetz isomorphism  is shown for mixed semi-simple perverse sheaves defined over a finite field in \cite[Thm.~5.4.10]{BBD82}  and
for semi-simple perverse sheaves of geometric origin  over the
complex numbers  in \cite[Thm.~6.2.10]{BBD82}.
 Combining this with the Langlands correspondence for function fields~\cite{Laf02} one can
show Theorem~\ref{thm:hl2} by first specializing to an algebraic closure of a finite field $F$ similarly
to~\cite[Rmk.~1.7]{Dri01} and then using the method of the proof of~\cite[1.8]{Dri01}.  In the
next two theorems, $X$ is a smooth projective variety over an algebraic closed field $F$ of
characteristic different from $\ell$, and $\eta\in H^2(X, \Z_\ell)$ is the first Chern class of an ample line bundle.

\begin{thm}[Hard Lefschetz] \label{thm:hl2}
 Let $\sF\in D^b_c(X, \overline \Q_\ell)$ be an arithmetic  semi-simple perverse sheaf.  Then for any $i\in \N,$ the cup-product map
 \[
\cup \eta^{ i}\colon H^{-i}(X, \sF ) \xrightarrow{\sim}   H^{i}(X, \sF)
\]
is an isomorphism.
\end{thm}

The aim of this section is to prove Theorem~\ref{thm:hl} for which weights are not available.

\begin{thm}[Hard Lefschetz]  \label{thm:hl}
Let $\sL$ be an \'etale rank one $\QQl$-local system on $X$.  Let $\sF\in D^b_c(X, \QQl)$ be an arithmetic  semi-simple perverse sheaf.
Then for any $i\in \N,$ the cup-product map 
 \[
\cup \eta^{ i}\colon H^{-i}(X, \sF\otimes \sL) \xrightarrow{\sim}   H^{i}(X, \sF\otimes \sL)
\]
is an isomorphism.
\end{thm}

\begin{rmk}
One conjectures the Hard Lefschetz isomorphism to hold for any semi-simple perverse
$\QQl$-sheaf on a projective scheme $X$ over an algebraically closed field $F$ of
characteristic different from $\ell$, see \cite{Dri01}.
\end{rmk}

\subsection{Proof of Hard Lefschetz}

We now prove  Theorem~\ref{thm:hl}.
We follow the general strategy of Drinfeld in \cite[Lem.~2.5]{Dri01}.

We first make the reduction to the case $F=\overline \F_p$.  To this aim we quote
\cite[Lem.~6.1.9]{BBD82} which unfortunately is only written for the passage from
the field of complex numbers to positive characteristic, so we quote in addition
\cite[Rmk.~1.7]{Dri01}  where it is observed that the spreading and specialization also
work in positive characteristic with only minor changes.

From now on, we assume that $F=\overline \F_p$. Then $X$ descends to a variety over a
finite subfield $F_0\subset F$.  With the notation of
Section~\ref{sec:formgr} we let $\pi$  be the $\ell $-adic completion of the abelian \'etale
fundamental group $\pi_1(X)^{\rm
  ab}$  modulo torsion.


We first show that if the Hard Lefschetz theorem is true for all $\sL$ which factor
through $\pi$, then it is true in general. Indeed,   write $\chi\colon \pi(X)^{\rm ab}\to
\sO_E^\times$ for the character corresponding to $\sL$ where $\pi_1(X)^{\rm ab}$ is
the maximal abelian quotient.  We can decomposte $\pi_1(X)^{\rm ab}$ into a product
$\pi\times \pi'$, where $\pi'$ is itself a product of a finite $\ell$-group and a
pro-finite prime to $\ell$ group. As $\chi(\pi')$ is finite this leads to a decomposition
$\sL = \sL' \otimes \sM$ such that $\sL'$ is a torsion rank one $\QQl$-local system and
such that the character of $\sM$ factors through $\pi$.
The sheaf  $\sF\otimes \sL'$  is still
semi-simple perverse and arithmetic,  so by our assumption, Hard Lefschetz holds for
$(\sF\otimes \sL' )\otimes \sM=\sF\otimes \sL$.

We now prove Hard Lefschetz under the assumption that the character of $\sL$ factors through $\pi$. 
We consider the multiplicative  formal Lie group
$\sG={\rm Spf} (\Z_\ell\llbracket \pi \rrbracket)$. So
$\sG(\QQl)$ parametrizes the $\QQl$-local systems whose character
 factors through $\pi$.
We define  $S_\circ\subset \sG(\QQl)$ to correspond to those $\QQl$-local systems  $\sL$ such that Theorem~\ref{thm:hl} fails. By
Corollary~\ref{cor:cons},  $S_\circ$ is constructible.
We define $S\subset \sG(\QQl)$ to be the Zariski closure of $S_\circ$.

The geometric Frobenius $\sigma\in {\rm Gal}(F/F_0)$ induces an automorphism   $\sigma\colon \sG\to \sG$  of the formal Lie group. As
$\sF$ is assumed to be arithmetic, we can replace $F_0$ by a finite extension and assume that  $\sigma$  fixes $\sF$ up to quasi-isomorphism. Thus we obtain a
(non-canonical) isomorphism
\[
\sigma \colon H^*(X,\sF\otimes \sL) \xrightarrow{\sim}  H^*(X,\sF\otimes \sigma(\sL))
\]
compatible with the cup-product with $\eta^i$.
So $\sigma (S_\circ) =S_\circ$ and $\sigma(S)=S$.

The Frobenius $\sigma$ acts $\Z_\ell$-linearly and semi-simply
on $\pi\otimes_{\Z_\ell}\Q_\ell$, use \cite[Thm.~2]{Tat66} and note that the Albanese map $X\to {\rm Alb}(X)$
induces an isomorphism with the Tate module $\pi\cong T_\ell({\rm Alb}(X))$~\cite[Ann.~II]{Ser58}.
Furthermore, the Frobenius $\sigma$ acts with  weight $-1$ on $\pi$, see \cite[Thm.~1]{Del80}.
Therefore Theorem~\ref{main.thm} is applicable and says that
$S$ is quasi-linear.

By Lemma~\ref{lem:torsiondense}, the torsion points are dense in $S$.
If $S_\circ$ were non-empty  it would contain a torsion point corresponding to an arithmetic rank one
$\QQl$-local system $\sL$.  But then $\sF\otimes \sL$ would be perverse, semi-simple,
arithmetic and Hard Lefschetz would fail for it. This contradicts Theorem~\ref{thm:hl2}.  So $S_\circ$ is empty. This finishes the proof.

\section{Jumping loci and generic vanishing}\label{sec:jumvan}

\noindent
In this section we discuss further applications of our main Theorem~\ref{main.thm}.

\subsection{Flat locus of the Fourier-Mellin transform and the cohomological jumping locus}\label{subsec:jump}

Let $X$ be a separated, connected scheme of finite type over the algebraically closed field
$F$ of characteristic different from $\ell$. Let $E$ be a finite extension of $\Q_\ell$
with a fixed embedding $E\hookrightarrow\QQl$.
Let $\pi$ be a finitely generated free $\Z_\ell$-module  which is  a quotient of $\pi_1^{\rm ab}(X)$. Set
$$R = \sO_E\llbracket \pi \rrbracket, \quad \mathfrak R=R\otimes_{\sO_E} \QQl$$ and let $\sG =
\Spf(\sO_E \llbracket \pi \rrbracket ) $ be
the associated multiplicative formal Lie group. For $\sF\in D^b_c(X,\QQl)$ and $i,j\in \Z$ consider the subset
 \[
\Sigma^i_\pi(\sF,j) := \{ \sL\in \sG(\QQl ) \, |\, {\rm dim} \ H^i(X ,\sF\otimes\sL ) > j \}.
\]
of $\sG(\QQl)$, where we omit the index $\pi$ if it is clear from the context.
It is Zariski closed as one easily sees from combining Proposition~\ref{prop:propertfm}
and \cite[Thm.~7.6.9]{EGAIII}.

As $\pi\otimes_{\Z_\ell} \Q_\ell$ is dual to a subgroup of $H^1(X,\Q_\ell)$, it has a
canonical weight filtration (\cite[Sec.~2]{Jan10}). The weight zero part of  $H^1(X,\Q_\ell)$ is
equal to the kernel  $H^1(X,\Q_\ell)\to  H^1(X^{\rm reg},\Q_\ell)$, where
$X^{\rm reg}$ is
the regular locus of the reduced scheme $X_{\rm red}$.

\begin{ex}\label{ex:weightdo}
  \begin{itemize}
    \item[(1)]
      For $X$ proper, integral and geometrically unibranch the group $H^1(X,\Q_\ell)$ is pure of weight one.
    \item[(2)]
  For $X\subset Y$ an open
      subscheme of a smooth variety $Y$ over $F$ with $H^1(Y,\Q_\ell)=0$
    the cohomology group $H^1(X,\Q_\ell) $ is pure of weight two. Indeed,
\[
  H^1(X,\Q_\ell) \to
   H^2_{Y\setminus
      X}(Y,\Q_\ell) = \Q_\ell(-1)^{\oplus b}
  \]
  is injective and the group on the right side 
  is pure of weight two, where $b$ is the number of irreducible components of $Y\setminus
  X$ which are of codimension one.
 This holds for example for $X=\G_m^d\subset Y=\P^d_F$.
\end{itemize}
\end{ex}

\begin{thm}\label{thm.jump}
  Assume that  $\pi$ is pure of weight different from zero.
  Let
  $\sF\in D^b_c(X,\QQl)$ be arithmetic.
  \begin{itemize}
  \item[(1)] Then
$\Sigma^i(\sF,j)$ is quasi-linear for all $i,j\in \Z$.
\item[(2)]
  The non-flat locus
  $S\subset \Spm(\mathfrak R )=\sG(\QQl )$ of the $\mathfrak R$-module $\sFM^i(X,\sF)$ is quasi-linear.
\end{itemize}
\end{thm}

\begin{rmk}
Conjecturally, the theorem remains true if we only assume that
$\pi$ is mixed of non-zero weights instead of pure. Indeed, for $\ch(F)=0$,  this is true
by~\cite[Thm.~1.5] {EK20} and the general case would follow along the same lines from a
``mixed version'' of~\cite{deJ00}.

If  $\sF\in D^b_c(X,\QQl)$ is not assumed to be arithmetic we can still conjecture the theorem
to hold without the part that the translation is by torsion points $s_r\in \sG(\QQl)$ in
the  Definition~\ref{def:ql}
of quasi-linearity.
 \end{rmk}

\begin{rmk}
 Results related  to Theorem~\ref{thm.jump}  for  $\ch(F)=0$ have been extensively  studied using complex analytic techniques,
  see \cite{GL91}, \cite[Thm.~3.1,Cor.~6.4]{Sim93}, \cite[Thm.~3.1]{Sab92}, \cite[Sec.~11]{KW15}  and~\cite[Thm.~1.1]{BW15}.
For the torus $X=\G_m^d$
part (1) of Theorem~\ref{thm.jump}  proves part of a conjecture of Loeser~\cite[Intro.~p.9]{Loe96} under our arithmeticity assumption.
\end{rmk}

\begin{proof}[Proof of Theorem~\ref{thm.jump}]
  For simplicity of notation we stick to part (1), as part (2) is proved almost verbatim
  the same way.
  By the same technique as in the proof of Hard Lefschetz in Section~\ref{sec:HL}, we can assume
  that $F$ is the algebraic closure of a finite field $F_0$ and that
the scheme $X$ descends to a scheme $X_0$ of finite type over $F_0$. By the arithmeticity
condition on $\sF$, after replacing $F_0$ by a finite extension, for any $\sigma\in {\Gal}(F/F_0)$
we have $\sF \simeq \sigma (\sF)\in D^b_c(X ,\QQl )$.
Without loss of generality we can assume that $\pi$ is the $\ell$-adic completion of
$\pi_1^{\rm ab}(X)$ modulo torsion. Then $ {\Gal}(F /F_0)$ acts on $\pi$. It follows that  $\Sigma^i(\sF,j)$ is stabilized by any $\sigma \in   {\Gal}(F/F_0)$ for $i,j\in \Z$.
We claim that we can apply
Theorem~\ref{main.thm} with $\sigma\in  {\Gal}(F /F_0)$ the Frobenius, to see that
$\Sigma^i(\sF,j)$ is quasi-linear. For this we have to see that the Frobenius $\sigma$
acts semi-simply on $\pi\otimes\Q_\ell$.

 We may assume that $X$ has dimension $\ge 1$. Choose a regular closed subscheme $Y\hookrightarrow X$ of dimension one such that the
composition $\pi_1(Y)\to\pi_1(X)\to\pi$ is surjective, for example by
using~\cite[App.~C]{Dri12}. By~\cite[Ann.~II]{Ser58} the  Albanese map  $\pi_1(Y)\to
T_\ell( {\rm Alb}(Y))$ identifies $T_\ell( {\rm Alb}(Y))$ with the $\ell$-adic completion
of $\pi_1(Y)$ modulo the torsion subgroup. So  $\pi\otimes_{\Z_\ell}\Q_\ell$ is isomorphic to a quotient  of the torus
part or of the abelian part of the Tate module $V_\ell( {\rm Alb}(Y))$.    By Tate~\cite[Thm.~2]{Tat66} the
Frobenius action  on the Tate module of an abelian variety is
semi-simple. The Frobenius action on the Tate module of the torus is a scalar
multiplication after replacing $F_0$ by a finite extension.
This finishes the proof.

\end{proof}

\subsection{Generic vanishing}\label{subsec:genvan}

This section is motivated by~\cite[Thm.~1.1]{BSS18}. Except for the proof of Theorem~\ref{thm.genvan},
what we say is only an $\ell$-adic translation of  {\it loc.cit. }
The classical question on   a lower bound for the codimension of the jumping loci
$\Sigma^i(\sF,0)$ has been initiated by Green-Lazarsfeld for line bundles~\cite{GL91}.

We
consider a smooth projective variety $X$ over the algebraically closed field $F$ of
characteristic different from $\ell$. As before $\pi$  is a  free $\Z_\ell$-module
quotient of $\pi_1^{\rm ab}(X )$. Our
criterion on generic vanishing depends on properties of the tower of Galois coverings
\[
\cdots \to X_{n+1} \to X_n \to \cdots \to X_0 =X
\]
  with
Galois groups $\Gal(X_n/X)=\pi/\ell^n\pi$. We then write $X_\infty$ for this tower and $H^i(X_\infty ,\Z_\ell )$ for $\colim_n
H^i(X_n , \Z_\ell )$. Let $\eta\in H^2(X , \Z_\ell )$ be a polarization, i.e.\ the
first Chern class of an ample line bundle, where we omit Tate twists for simplicity of notation.
We consider a situation in which $\eta $ becomes {\it divisible} in $H^2(X_\infty
,\Z_\ell)$. Note that a class
 $z\in  H^i(X, \Z_\ell)$
becomes divisible in $H^i(X_\infty, \Z_\ell)$
  if
\[\forall m, \exists n \ \text{such  that} \  p_n^* ( z) \in \ell^m H^i(X_n, \Z_\ell) \subset H^i(X_n, \Z_\ell),\]
where $p_n \colon X_n\to X$ is the covering map.

We describe now two kinds of Examples~\ref{eq.tower10} and~\ref{eq.tower20}.

\begin{ex}\label{eq.tower10}
Let $X$ be an abelian variety of dimension $d$ with
  polarization $\eta$. Let $V$ be a
  maximal isotropic subspace of the Tate module $T_\ell(X)\otimes_{\Z_\ell} \Q_\ell$   with respect to the Weil
  pairing associated to $\eta$, so $\dim_{\Q_\ell}(V)=d$. Note that $T_\ell(X)$   can be identified with  the $\ell$-adic completion of
  $\pi_1^{\rm ab}(X)$.
 We set $\pi=T_\ell(X)/
  T_\ell(X)\cap V$ which is a free  $\Z_\ell$-module of rank $d$.  Then $\eta$ becomes   divisible in $H^2(X_\infty ,\Z_\ell
  )$ by~\cite[Thm.~23.3]{Mum70}.
\end{ex}

\begin{ex}\label{eq.tower20}
Let $X=X^{(1)}\times \cdots \times X^{(d)}$ be a product of smooth proper curves with polarizations
$\eta^{(1)},\ldots , \eta^{(d)}$ and let  $ X_\infty^{(j)}\to X^{(j)}$ be towers with Galois group
$\Z_\ell$. We consider the tower $ X_n =  X^{(1)}_n      \times \cdots \times X^{(d)}_n$
and the polarization $\eta= \eta^{(1)} + \cdots + \eta^{(d)}$ of $X$.  Then $\eta$ becomes
divisible in  $H^2(X_\infty ,\Z_\ell )$.
\end{ex}

\begin{rmk}
 The tower in Example~\ref{eq.tower20} is used in \cite[Prop.4.5.1]{BBD82} to prove via  \cite[Cor.4.5.5]{BBD82}
 the purity theorem for
 the intermediate perverse extension in \cite[Cor.~5.3.3]{BBD82}.
\end{rmk}

\begin{thm}\label{thm.genvan}
Let $X$ be a smooth projective variety over the algebraically closed field $F$.
  Assume that $\sF \in {}^p\! D^{\le 0}_c(X,\QQl )$ is arithmetic and that the pullback to $H^2(X_\infty , \Z_\ell )$
  of a polarization $\eta \in H^2(X , \Z_\ell )$
  becomes divisible. Then
    \[
\codim_{\Spm (\mathfrak R )}(\Sigma^i(\sF,0)) \ge i
\]
for all $i\ge 0$.
\end{thm}

The proof of Theorem~\ref{thm.genvan} is given at the end of this section. 
  It uses Proposition~\ref{prop.keyvanish}, Theorem~\ref{thm:hl} and  Theorem~\ref{thm.jump}.

\begin{cor}\label{cor:genvan1}
For  $\sF \in{} ^{p}\! D^{\ge 0}_c(X,\QQl )$  the Fourier-Mellin transform satisfies
  $\sFM^i(X,\sF)=0$ for $i<0$.
\end{cor}

\begin{proof}
Using Proposition~\ref{prop:propertfm}(4)  on duality together with
Remark~\ref{rmk:fmdual}, we see that
 \[
  \sFM(X,\sF)' \simeq R\Hom_{\mathfrak R}(
  \sFM(X,\sF^\vee) , \mathfrak R).\]
  We prove by descending induction on $i$ that
  \begin{equation}\label{eq:codimfm}
\codim\big({\rm supp} ( \sFM^i(X,\sF^\vee) )\big) \ge i .
\end{equation}
In fact,  Corollary~\ref{cor:genvan1} follows from~\eqref{eq:codimfm} and the above
duality  by  support estimates for dual complexes~\cite[Lem.~2.8]{BSS18}.

For $i>\dim(X)$ we have $ \sFM^i(X,\sF^\vee)=0$ for cohomological dimension reasons that,
i.e.\ ${\rm supp} ( \sFM^i(X,\sF^\vee) )$ is empty. This starts the induction.
Now we fix $i>0$ and we assume that~\eqref{eq:codimfm} is known in degrees bigger than $i$. From
Proposition~\ref{prop:propertfm}(2) and the Tor-spectral sequence we deduce that for $s\notin {\rm supp} (
\sFM^j(X,\sF^\vee))$ for all $j>i$, i.e.\ for $s$ outside of a closed subset of codimension $\ge i+1$, we have
\[
  \sFM^i(X,\sF^\vee)\otimes_{\mathfrak R} k(s) \cong H^i(X,\sF^\vee \otimes \sL_s ).
\]
The right hand side vanishes for $s$ outside of a closed subset of codimension $ \ge i$
by  Theorem~\ref{thm.genvan}. This shows~\eqref{eq:codimfm}.
  \end{proof}

\begin{cor}\label{cor:genvan}
  Assume that $X$ is an abelian variety and that $\pi=T_\ell(X)$. Let $\sF \in
{}^p\!  D^{\le 0}_c(X,\QQl )$ be arithmetic.
 Then
    \[
\codim_{\Spm (\mathfrak R )}(\Sigma^i(\sF,0)) \ge 2i
\]
for all $i\ge 0$.
\end{cor}

\begin{proof}
By Theorem~\ref{thm.jump} we know that up to a translation by a torsion character, which
we can assume to be trivial in the following,
each irreducible component $S$ of $ \Sigma^i_\pi(\sF,0)$ corresponds to a free $\Z_\ell$-module quotient
$\pi/\pi'$ of $\pi$. In particular,  $\codim (S)={\rm rank} (\pi')$.
 Let $\pi^{\rm half}$ be a quotient of $\pi=T_\ell(X)$  by an isotropic
subgroup $V\cap \pi$ as in Example~\ref{eq.tower10}, where $V$ is chosen  such that
  $2\, {\rm rank} (\pi'\cap V)\ge {\rm rank}  (\pi')$. Clearly, ${\rm
    rank}(\pi^{\rm half})= \dim(X)$.  We have an exact sequence
\ga{}{ 0\to V\cap \pi'\to \pi' \to {\rm im}(\pi'\to \pi^{\rm half})\to 0.\notag
}
So 
\[
{\rm rank}(\pi')/2\ge  {\rm rank} \left( \im (\pi'\to \pi^{\rm half} ) \right)  \ge {\rm codim} (\Sigma^i_{\pi^{\rm half}}(\sF,0))    \ge i.
\]
Here the right inequality follows from Theorem~\ref{thm.genvan} applied to the setup of
Example~\ref{eq.tower10} where $X_\infty/X$ has automorphism group $\pi^{\rm half}$. 
\end{proof}

\begin{rmk}\label{rmk:genvan1}
One can expect the theorem and its corollaries to hold for non-arithmetic $\sF$. In fact,   for $\ch(F)=0$ this
 can be shown by essentially the same technique using that  Hard Lefschetz is known
in general~\cite[Lem.2.6]{Sim92}  (and more generally \cite[Cor.1.1]{Moc07}) and using \cite[Thm.~3.1 (c)]{Sim93} instead of Theorem~\ref{main.thm}.
\end{rmk}

\begin{rmk}
  For $\ch(F)=0$ Corollary~\ref{cor:genvan} is equivalent
to  ~\cite[Thm.~4.1]{Sch15} and \cite[Thm.~1.3]{BSS18}, who do not
need the arithmeticity assumption, compare Remark~\ref{rmk:genvan1}, see also~\cite{KW15}.
For $\ch(F)>0$,  the  inequality    $\codim(\Sigma^i(\sF,0)) >0$ for $i> 0$ is shown in \cite[Intro]{Wei16}.
\end{rmk}

\begin{rmk}
In case of Example~\ref{eq.tower20}, the  theorem and its corollaries can be easily deduced
from \cite[Thm.~1.1]{Esn21}, even for
non-arithmetic $\sF$. In fact in this
situation we deduce from {\it loc. cit.} that Corollary~\ref{cor:genvan1} even holds integrally, i.e.\ for $\sF\in {}
^p \!D^{\ge 0}_c(X,\F_\ell)$ we have $H^i(X, \sF \otimes_{\F_\ell } \sL_{\F_\ell \llbracket
\pi \rrbracket } )=0$ for $i<0$. Bhatt--Schnell--Scholze ask \cite[Rmk.~2.11]{BSS18}  whether the analog of the
latter integral result is true in the situation of Example~\ref{eq.tower10}.

\end{rmk}

In the proof of Theorem~\ref{thm.genvan} the following lemma is crucial.

\begin{lem} \label{lem:crucial}
Let $\pi'$ is a closed subgroup of $\pi$ with $\pi/\pi'$ torsion free.  Let
 $X'_\infty \to X$ be the Galois tower of $\pi/\pi'$ and  $r={\rm rank}  ( \pi' )+1$. Then $\eta^r$ becomes divisible
in $H^{2r}(X'_\infty, \Z_\ell)$.
\end{lem}

\begin{proof}
  The lemma is equivalent to saying that for all $m>0$,  the image of $\eta^r$ in $H^{2r}(X'_\infty,
  \Z/\ell^m)$ vanishes.
 The Hochschild-Serre spectral sequence of the covering $X_\infty \to X'_\infty$ yields a
 spectral sequence
 \[
E^{pq}_2 = H^p(\pi', H^q(X_\infty , \Z /\ell^m  ) ) \Rightarrow H^{p+q}(X'_\infty,  \Z /\ell^m ),
\]
with associated filtration $F^* H^{p+q}(X'_\infty),$ which is compatible with cup-product. As $\eta $ vanishes in
$$E_2^{0  2}\subset H^2(X_\infty,\Z /\ell^m ), $$ the image of $\eta$ in $
H^{2}(X'_\infty,  \Z /\ell^m)$ lies in $ F^1
H^{2}(X'_\infty,  \Z /\ell^m)$. So the image of $\eta^r $ lies in $ F^r
H^{2r}(X'_\infty,  \Z /\ell^m) .$  As
 ${\rm cd}_\ell (\pi')=r-1$, $E^{pq}_2=0$ for $p>r-1$. Thus
 $$ F^r
H^{2r}(X'_\infty,  \Z /\ell^m)=0 .$$
This finishes the proof.
\end{proof}

\begin{proof}[Proof of Theorem~\ref{thm.genvan}]
  If $\sF_1 \to \sF_2 \to \sF_3 \to \sF_1[1]$ is an exact triangle in $D^b_c(X,\QQl)$ such that
  Theorem~\ref{thm.genvan} holds for $\sF=\sF_1$ and for $\sF=\sF_3$ then it also holds
  for $\sF_2$.
By~\cite[Thm.~4.3.1]{BBD82}  we can therefore assume that $\sF$ is a
simple perverse sheaf.

Assume that the theorem fails in some degree $i>0$. Let $S$
be an irreducible component of $\Sigma^i(\sF,0)$ of codimension $<i$ in $\Spm(\sR)=\sG(\QQl)$. By Theorem~\ref{thm.jump}(1) we see that
$S$ is the translation by a torsion  point $s$ of a closed subset
of the form $\Spm(\mathfrak R' ) $, where $\mathfrak R' = \Z_\ell \llbracket
\pi/\pi'\rrbracket  \otimes_{\Z_\ell} \QQl $. Replacing $\sF$ by $\sF\otimes \sL_s$ we can assume without
loss of generality that $s=1$, i.e.\  $S= \Spm (\mathfrak R')$.

In terms of the Fourier-Mellin transform we see that
\begin{equation} \label{eq.fmnonvan}
\sFM_{\pi/\pi'}^i(X,\sF)\otimes_{\mathfrak R'} {\rm Frac}({\mathfrak R'}) \ne 0,
\end{equation}
because for a generic point $s\in \Spm (\mathfrak R' )$ we have   \[
\sFM_{\pi/\pi'}^i(X,\sF)\otimes_{\mathfrak R'} k(s)\cong H^i(X, \sF\otimes \sL_s )\] by
Proposition~\ref{prop:propertfm} and because the cohomology group on the right does not vanish.

By Theorem~\ref{thm:hl} and Lemma~\ref{lem.geniso} we see that the Lefschetz map
\begin{equation}\label{eq:fmnonvan2}
\sFM_{\pi/\pi'}^{-i}(X,\sF) \xrightarrow{\cup  \eta^i} \sFM_{\pi/\pi'}^i(X,\sF)
\end{equation}
 becomes an isomorphism when tensored with  $ {\rm Frac}({\mathfrak R'})$, so the map~\eqref{eq:fmnonvan2} is non-vanishing according
to~\eqref{eq.fmnonvan}. But by Proposition~\ref{prop.keyvanish} and Lemma~\ref{lem:crucial}
we see that the map~\eqref{eq:fmnonvan2} vanishes, since  ${\rm rank}(\pi')= \codim (S) <i$. This
is a contradiction and  finishes the proof.
\end{proof}

\end{document}